\documentclass[12pt]{amsart}
\usepackage[margin=1.4in]{geometry}
\usepackage[english]{babel}
\usepackage[utf8]{inputenc}
\usepackage{subcaption}
\usepackage{amsmath}
\usepackage{amssymb}
\usepackage{amsfonts}
\usepackage{bbm}
\usepackage{amsthm}
\usepackage{mathrsfs}
\usepackage[all]{xy}
\usepackage[pdftex]{graphicx}
\usepackage{color}
\usepackage{cite}
\usepackage{url}
\usepackage{indent first}
\usepackage[labelfont=bf,labelsep=period,justification=raggedright]{caption}
\usepackage[english]{babel}
\usepackage[utf8]{inputenc}
\usepackage{hyperref}
\usepackage[colorinlistoftodos]{todonotes}
\usepackage{tkz-fct}
\usepackage{tikz}
\usetikzlibrary{calc}
\usepackage{multicol}
\PassOptionsToPackage{dvipsnames,svgnames}{xcolor}
\usepackage{textcomp}
\usetikzlibrary{decorations.markings}

\def\multiset#1#2{\ensuremath{\left(\kern-.3em\left(\genfrac{}{}{0pt}{}{#1}{#2}\right)\kern-.3em\right)}}

\theoremstyle{plain}
\newtheorem{thm}{Theorem}
\newtheorem{lemma}[thm]{Lemma}
\newtheorem{cor}[thm]{Corollary}

\theoremstyle{definition}
\newtheorem{defn}[thm]{Definition}

\theoremstyle{remark}
\newtheorem{rem}[thm]{Remark}
\newtheorem*{ex}{Example}

\numberwithin{equation}{section}
\numberwithin{thm}{section}
\begin{document}
\title{MARTINGALES}
\bigskip
\author{ROHAN SHAH}
\bigskip
\date{\today}
\address{Milton Academy, Milton, MA, 02186}
\email{rohan14shah@gmail.com}
\maketitle
\begin{center}
Abstract. This paper covers martingales with introduction to measure theory concepts, and other concepts, including the Lebesgue Integration and Conditional Expectation. It follows up with proofs on Kolmogorov’s Theorem on conditional expectations, the Martingale Property, and the Pythagorean Theorem on Martingales. Finally, it ends with martingales' application in finance. 
\end{center}
\bigskip
\renewcommand*\contentsname{\textbf{CONTENTS}}
\tableofcontents
\bigskip
\section{INTRODUCTION}
Let's consider a bettor participating in a game involving the flipping of a fair coin. In this game, the bettor earns \$1 for a heads outcome and loses \$1 for a tails outcome. This means that if the better flips the coin once and it lands on tails, they would lose a dollar. Conversely, if it lands on heads, they would gain a dollar.

Over many run-throughs of flipping the coin, the bettor’s total earnings, represented as T, could be either positive or negative, depending on the outcomes of the flips. Each individual coin flip carries an equal probability of 50\%, meaning there is an equal chance that the next flip will result in either heads or tails. Consequently, with each flip, the bettor’s total earnings will either increase by \$1 (if heads) or decrease by \$1 (if tails).

After many flips, the bettor’s total earnings will result from random fluctuations due to the equal probability of heads or tails. Over many flips, the total earnings T will show a random pattern of gains and losses, yet the expectation of total earnings remains zero. This means that at one point, despite the variances in time, basic probability states that the final value will be 0. This is an example of a fair game, where the expected earnings after the next turn would equal the current earnings. 
A martingale shares a similar idea. Introduced by Paul Lévy in 1934 and later by Ville in 1939, a martingale is a probability model based on a fair game. This property can be translated to concepts outside regular games, including the stock market. In simple terms, a martingale is a process where the conditional expectation of the next value is equal to the present value. Since there are no predictable trends in a fair game, a martingale can effectively visualize the randomness and fairness of such games. This paper covers the mathematical definition of martingale with the requirement of knowledge of elementary probability, specifically measure theory ($\sigma$-Algebra, measurable spaces, Lebesgue measure, Borel sets). 

Before we begin, I will introduce the notation for a martingale and explain each part throughout the paper.

\begin{defn}
    $X$ is a martingale if
    $$
    X_n = E[X_{n+1}|\mathcal{F}_n]
    $$
\end{defn}

\section{PROBABILITY AND MEASURE THEORY}
We begin with definitions regarding the basics of measure theory, including sets and events. We are given a sample space denoted by $\Omega$ as the set of all possible outcomes of a random experiment. An event is a subset of the sample space $\Omega$. For instance, if $\Omega$ represents the possible outcomes of a dice roll, then (2,4,6) could be an event representing the roll of an even number. We begin by looking at the first approximation, the power set of omega, $P(\Omega)$ which equals the set of all subsets of $\Omega$. To restrict this power set, we introduce our first definition. 
\begin{defn}
   A system $A \subseteq P(X)$, where A is a collection of elements of subsets of X. This collection is called $\sigma$-algebra. A $\sigma$-algebra is a collection of subsets of the sample space that includes the empty set and the sample space itself. However, to be named a $\sigma$-algebra (an important measure theory definition), it must follow these three rules:
\begin{itemize}
    \item $\varnothing, X \subseteq P(X)$
    \item If $B \in A$ then $B^C = X/ B \in A$. This means that for any set in the $\sigma$-algebra, its complement is also in the $\sigma$-algebra.
    \item $B,D \in A \to B\cup D \in A$. Additionally, if the set B and D are in set, A, the intersection of set B and D will also be in A. 
    \item Then $(X,A)$ is a measurable space.
\end{itemize} 
\end{defn}
Measurable spaces are essential to provide a framework in which measures are defined. In the definition above, $(X,A)$ is a measurable space with set X and the collection of subsets, the sigma-algebra, A. The sigma-algebra includes the subsets of X, that are considered measurable.
\begin{defn}\ 
    \begin{itemize}
    \item For $M \subseteq P(X)$, there is a smallest sigma algebra that contains M: 
    \begin{center}
        $\sigma(M) =: \bigcap_{\substack{A \supseteq M}} |$ $A$ is a $\sigma$-algebra generated by set M.
    \end{center}
    \item E.g. $X=\{a, b, c, d\}, M=\{\{a\},\{b\}\}$
    \item Therefore, $\sigma(\mathcal{M})=\left\{\begin{array}{r}\phi, X,\{a\},\{b\},\{a, b\}, \{b, c, d\},\{a, c, d\},\{c, d\}\}
\end{array}\right.$
\end{itemize}
\end{defn}
By defining $\sigma(M)$ as the intersection of all sigma-algebras containing $M$, we guarantee that $(\sigma(M))$ includes all necessary sets to be a sigma-algebra while being minimal in the sense that it does not include any extraneous sets not required by the sigma-algebra properties and the inclusion of $M$.
\begin{defn}
    Let A be an algebra, $\mu: A \to [0,\infty]$ is a finitely additive measure on A, if 
    \begin{itemize}
    \item For any two sets $B,D \in A$ with $B \cap D = \varnothing: \mu(B\cup D) = \mu(B) + \mu(D)$. 
    \end{itemize} 
\end{defn}
\noindent
A finitely additive measure is a function that assigns a non-negative value to each set, representing its size, in a way that the measure of the union of two non-overlapping sets equals the sum of their measures. This property ensures consistent and additive measurements, allowing us to determine the size of complex shapes by summing the measures of their non-overlapping parts.
\begin{proof}
    \textit{Using Definition 2.1 and 2.3: }
    \begin{itemize}
    \item Then $A_1, A_2, A_3 \cdots A_n \in A: A_1 \cup A_2 \cup \cdots A_n \in A$
    \item If $A_i \cup A_j = \varnothing 1 < i \neq j < n$ , then $\mu \bigcup_{i = 1}^ n A_i$ = $\sum_{i=1}^n \mu (A_i)$. 
    \end{itemize}
\end{proof}
    
\noindent
The proof shows that any finite union of sets in a sigma-algebra $A$ is also in $\mathcal{A}$. It then uses the countable additivity property of measures to show that the measure of a finite union of disjoint sets is the sum of the measures of the individual sets.

\begin{defn}
    A measure $\mu$ on an algebra, A is a set function from  $A \to [0,\infty]$ if it follows the following properties:
    \begin{itemize}
    \item If we take any sequence of sets: $A_1, A_2, A_3 \cdots A_n \in A$ such that $\cup_{i = 1}^ n A_i$ and $A_i \cap A_j = \varnothing$ if $i\neq j$ then $\mu (\cup_i A_i) = \sum_i \mu(A_i)$.
\end{itemize}
\end{defn}
\noindent
Note that this property is denoted as $\sigma$-additivity. However if the set function, A is in the interval $[0,1]$, then these properties won't apply as it is a probability measure rather than a measure. Now, we can define what a probability measure is.
\begin{defn}
    A probability measure (P) on the sample space satisfies $P(\Omega)=1$. Therefore, the triple $(\Omega, \mathcal{F}, P)$ is called a probability space, where $\mathcal{F}$ is the sigma algebra and $P$ is the probability measure.
\end{defn} 
\noindent
With these terms defined, we can introduce what random variables are. 
\begin{defn}
   Given a measurable space $(\Omega, \mathcal{F})$, a function $X: \Omega \to \mathbb{R}$ is said to be a measurable function or a random variable if for every Borel Set, B $\in \mathbb{B}(\mathbb{R})$, the pre-image of that set, $X^{-1} (B) \in \mathcal{F}$. This implies that the function $X$ maps events in the sample space $\Omega$ to real numbers in a way that preserves the structure of $\sigma$-algebra. 
\end{defn}
\begin{defn}
    A random variable is a measurable function from the sample space to the real numbers. The distribution of a random variable $X$ is a function $X: B\to[0,1]$ such that $X(B)=P(X^{-1}(B))$ for all Borel sets $B$. Note that Borel sets (subset) are important as a way of showing that if open sets are measurable, then so are all other sets of interest (like closed sets, countable unions, and intersections).
\end{defn}

\section{LEBESGUE INTEGRATION}

If X, a integrable random variable takes on an infinite number of values, calculating the weighted average requires integration to handle this infinite sum. However, the standard Riemann integral is often unsuitable in this context. The Riemann integral approximates the integral by splitting the domain into contiguous intervals to form "rectangles," then summing the areas of these rectangles. This approach might not make sense for an abstract sample space, and some measurable functions may not meet the continuity requirements for Riemann integration. Therefore, we use a more generalized form of the Riemann integral called the Lebesgue integral. However, before we define the integral, we need understand simple functions. 
\begin{defn}
    Let $A$ be a subset of $\Omega$. The indicator function of $A$ is the function $\mathbf{1}_A: \Omega \rightarrow \mathbb{R}$ defined as:
$$
\mathbf{1}_A(\omega)=\left\{\begin{array}{lll}
1 & \text { if } & \omega \in A \\
0 & \text { if } & \omega \notin A
\end{array} .\right.
$$
\end{defn}
\begin{thm}
   The indicator function of $A$ is measurable in every $\sigma$-algebra $\mathcal{F}$ containing $A$.
  
\end{thm}
\begin{proof}
    Every possible Borel set $B$ falls into one of four categories:
    \begin{itemize}
        \item $1 \in B, 0 \notin B$. Then $\mathbf{1}_A^{-1}(B)=A$;
        \item $0 \in B, 1 \notin B$. Then $\mathbf{1}_A^{-1}(B)=A^c$;
        \item $0,1 \in B$. Then $\mathbf{1}_A^{-1}(B)=\Omega$;
        \item $0,1 \notin B$. Then $\mathbf{1}_A^{-1}(B)=\varnothing$.
    \end{itemize}  
By definition, if $A \in \mathcal{F}$, then $\Omega, \varnothing, A^c \in \mathcal{F}$. Thus, by Definition $3.1, \mathbf{1}_A$ is measurable.
\end{proof}

\begin{defn}
    Simple functions are measurable functions that take on a finite number of values. They can be written as: $$g = \sum_{i=1}^n a_i \mathcal{X}A_i$$ 
\end{defn}

In this notation above, $a_i$ are real numbers, $A_i$ are measurable sets, and ${}_{\mathcal{X}A_i}$.

\begin{proof}
    If we integrate the following simple function using the linearity property:\\
    1. Define the Simple Function:
    $$
    g=\sum_{i=1}^n a_i \mathcal{X}_{A_i}
    $$
    2. Set Up the Integral:
    To integrate $g$ with respect to a measure $\mu$, we use the definition:
    $$
    \int g d \mu=\int \sum_{i=1}^n a_i \mathcal{X}_{A_i} d \mu
    $$
    3. Use Linearity of the Integral:
    By the linearity property of the integral, we can move the summation outside the integral:
    $$
    \int \sum_{i=1}^n a_i \mathcal{X}_{A_i} d \mu=\sum_{i=1}^n a_i \int \mathcal{X}_{A_i} d \mu
    $$
    4. Evaluate Each Integral:
    The integral of the characteristic function $\mathcal{X}_{A_i}$ of a set $A_i$ with respect to $\mu$ is simply the measure of $A_i$ :
    $$
    \int \mathcal{X}_{A_i} d \mu=\mu\left(A_i\right)
    $$
    5. Combine the Results:
    Substitute the measure of each $A_i$ into the sum:
    $$
    \int g d \mu=\sum_{i=1}^n a_i \mu\left(A_i\right)
    $$

    So, the integral of the simple function $g$ with respect to the measure $\mu$ is:
    $$\int \mathrm{gd} \mu=\sum_{\mathrm{i}=1}^{\mathrm{n}} a_i \mu(A i), \text { where } \mu \text { is a measure on } \Omega $$.

\end{proof}
\noindent
\textbf{Properties of Lebesgue Integral:}
\begin{itemize}
    \item Linearity: $\int(\mathrm{af}+\mathrm{bg}) \mathrm{d} \mu=\mathrm{a} \int \mathrm{fd} \mu+\mathrm{b} \int \mathrm{gd} \mu$
    \item Monotonicity: If $\mathrm{f} \leq \mathrm{g}$, then $\int \mathrm{fd} \mu \leq \int \mathrm{gd} \mu$
    \item Countable Additivity: $\int \sum_{i=1}^{\infty} f_i \mathrm{~d} \mu=\sum_{i=1}^{\infty} \int f_i \mathrm{~d} \mu$
\end{itemize}
\bigskip
\noindent
The utility of the Lebesgue integral becomes clearer with this definition. The Lebesgue integral of a simple function is simply the sum of its possible values weighted by the measure of the subsets of the domain corresponding to those values. In the context of a probability measure, this is equivalent to averaging the possible values of a random variable weighted by the probability of the events that produce those values—that is, the expectation. For more complex, non-simple functions, the integral can be defined by approximating the function with simple functions. This leads to the following theorem, which facilitates this process.

\begin{thm}
     Let $f: \Omega \rightarrow[0, \infty]$ be a measurable function on the space $(\Omega, \mathcal{F}, \mu)$. Then, there exists an increasing sequence of simple functions $\left(x_i\right)_{i=1}^{\infty}$ that converges point-wise to $f$.

\end{thm}
\noindent
This is known as the Simple Function Approximation Theorem. As we covered a little about expectation, we can know understand both that and conditional expectation. 

\section{CONDITIONAL EXPECTATION}
Conditional expectation, denoted $E[Y|X]$, represents the expected value of a random variable $Y$ given the knowledge of another variable $X$. It is used to update our expectation of $Y$ based on the additional information provided by $X$. To understand this better, we begin by defining simple random variables, then move to expectation. 
\begin{defn}
   A random variable, X is simple if there exists $n>0, X_1, X_2, \cdots, X_n \in \mathbb{R}$, $A_1, A_2, \ldots, A_n \in \mathcal{F}:$
$$
X=\sum_{Xi=1}^n X_{\varepsilon} \mathbbm{1}_{A_{\varepsilon}}
$$ 
\end{defn}
\noindent
This means that X is a function from $\Omega \to \mathbb{R}$. $A_{\varepsilon}$ are sets in the $\sigma$-algebra, meaning that they are subsets in $\Omega$.
\begin{thm}
    For every random variable X, there exist $X_1, X_2, \cdots$ that are known as simple random variables such that:
    $$
    |X_n| \leq |X|
    $$
    For every $\omega \in$ sample space, $\Omega$. 
    $$X_n(w) \vec{n}_\infty X(w)$$
\end{thm}
\begin{thm}
    If $X(\omega) \geq 0$ for all $\omega \in \Omega$, then $X_n(\omega)$ can chosen to be non-decreasing in $n$ for every $\omega$.
\end{thm}

\noindent
"Non-decreasing for n" means that as the sequence progresses, the values do not decrease for any outcome in the sample space. In other words, the values either stay the same or increase as the index increases, ensuring the sequence never gets smaller for each outcome.
\bigskip

\noindent
After these definitions of random variables, we can now define the expectation of a random variable.

\begin{defn}
    If $X=\sum_{Xi=1}^n X_{\varepsilon} \mathbbm{1}_{A_{\varepsilon}}$ is simple, then $\mathbb{E}X$ (expected value) $=\sum_{Xi=1}^n X_{\varepsilon} \mathbb{P}(A_{\varepsilon})$
\end{defn}

\noindent
This is Step 1, only applying if the random variable, X is a simple, following the properties defined in Definition 3.1. Now, we will discuss Step 2 when $X \geq 0$ (not negative).

\begin{defn}
    If $X \geq 0$ and a random variable, then $X_1, X_2, \ldots$ is a simple random variable such that $X_i$ converges to $X$. 
    $$\mathbb{E}X = lim_{X \to i} \mathbb{E}X_i \in [0,\infty]
    $$
\end{defn}
\begin{defn}
    If $X$ is a random variable, then $X := X^+ - X^- $
\end{defn}

\noindent
This second part of the definition is stating that the positive part of the variable, $X^+$ plus the negative part of the variable, $X^-$ equals to the variable itself. For instant the positive part of 4 is 4, but the negative part of 4 is 0. Adding these two numbers, we get the number 4 itself. This may seem intuitive, but if we do this for every $\omega$, we can define a function, where both the positive and negative part of the variable is not negative. If we take the expectation of this equation we get the following:
$$
\mathbb{E}X := \mathbb{E}X^+ - \mathbb{E}X^-$$ unless 
$$ \mathbb{E}X^+ = \infty = \mathbb{E}X^-
$$
Making $\mathbb{E}[X]$ not exist, which is known as the Cauchy random variable.\\
\noindent
Now we move to defining what conditional expectation is. 
\begin{thm}
    Kolmogorov's Theorem on conditional expectations.\\
    Let $(\Omega, \mathcal{F}, \mathbb{P})$ be a probability space, $X$ is an integrable random variable; $\mathbb{E}|X| < \infty$. Let $\mathcal{G} \subseteq \mathcal{F}$ be a sub sigma-algebra. Take random variable $Y$ such that it follows these properties: 
    \begin{itemize}
        \item $Y$ is $\mathcal{G}$-measurable. 
        \item $\mathbb{E}|Y| < \infty$
        \item $\forall A \in \mathcal{G}$
        \item $\int_A X dP = \int_A Y dP = \int_A \mathbb{E}(X|\mathcal{G}) dP$
    \end{itemize}
    The random variable $Y$ is denoted by $\mathbb{E}(X|\mathcal{G})$ and is called the conditional expectation of $X$ given $\mathcal{G}$. 
\end{thm}

The best way to explain this is through a proof, where we shall require the use of Radon-Nikodym Theorem. 
\begin{proof}
    Define a new measure, $Q$ on $(\Omega, \mathcal{G})$ by: $$Q(A) = \int_A X dP$$ for all $A \in \mathcal{G}$.\\
    Since $X$ is integrable, $Q$ is a finite measure on $(\Omega, \mathcal{G})$: $$Q(\Omega) = \int_\Omega X dP = \mathbb{E}[X] < \infty$$ \\
    The Radon-Nikodym theorem states that since $Q$ is a sigma-finite measure on $(\Omega, \mathcal{G})$ and $\left.P\right|_{\mathcal{G}}$ is the restriction of the probability measure $P$ to $\mathcal{G}$, there exists a $\mathcal{G}$-measurable function $Y$ such that:
    $$
    Q(A)=\int_A Y d P \quad \text { for all } A \in \mathcal{G}
    $$
    From here, since $A \in \mathcal{G}$: $$\int_A X dP=\int_A Y dP$$\\
    Define $Y=\mathbb{E}[X \mid \mathcal{G}]$. By construction, $Y$ is $\mathcal{G}$-measurable and satisfies:
    $$\int_A X d P=\int_A \mathbb{E}[X \mid \mathcal{G}] d P \quad \text {for all} A \in \mathcal{G}$$\\
    To show uniqueness, suppose $Y^{\prime}$ is another $\mathcal{G}$-measurable function that satisfies the same integral property. Then for all $A \in \mathcal{G}$ :
    $$
    \int_A \mathbb{E}[X \mid \mathcal{G}] d P=\int_A Y^{\prime} d P
    $$\\
    Taking $A=\left\{\mathbb{E}[X \mid \mathcal{G}]>Y^{\prime}\right\}$, we get:
    $$
    \int_{\left\{\mathbb{E}[X \mid \mathcal{G}]>Y^{\prime}\right\}}\left(\mathbb{E}[X \mid \mathcal{G}]-Y^{\prime}\right) d P=0
    $$\\
    Since $\mathbb{E}[X \mid \mathcal{G}]-Y^{\prime} \geq 0$ on $\left\{\mathbb{E}[X \mid \mathcal{G}]>Y^{\prime}\right\}$, this implies $P\left(\left\{\mathbb{E}[X \mid \mathcal{G}]>Y^{\prime}\right\}\right)=$ 0. A similar argument holds for $\left\{\mathbb{E}[X \mid \mathcal{G}]<Y^{\prime}\right\}$. Thus, $\mathbb{E}[X \mid \mathcal{G}]=Y^{\prime}$ almost surely.\\
    The Radon-Nikodym theorem provides a $\mathcal{G}$-measurable function $Y$ that satisfies the integral property, ensuring the existence of the conditional expectation $\mathbb{E}[X \mid \mathcal{G}]$. The uniqueness follows from the fact that any two $\mathcal{G}$-measurable functions satisfying this property must be equal almost surely. Thus, Kolmogorov's theorem on conditional expectations is proved.
\end{proof}

Let's move on to an important property of Conditional Expectation, the tower rule. 

\begin{defn}
    $(\Omega, \mathcal{F}, \mathbb{P})$; $\mathcal{G} \subseteq \mathcal{H} \subseteq \mathcal{F}$ (These are all $\sigma$-algebras). Both $\mathcal{G}$ and $\mathcal{H}$ are sub-sigma-algebras. Let $X$ be a random variable, $\mathbb{E}[X] < \infty$.\\
    Then the Tower rule states: $$\mathbb{E}(\mathbb{E}(X | \mathcal{G}) | \mathcal{H}) = \mathbb{E}[X|\mathcal{G}]$$ 
\end{defn}
\bigskip
\noindent
The Tower Rule essentially tells us that if we first condition on a larger sigma-algebra $\mathcal{G}$ and then on a smaller sigma-algebra $\mathcal{H}$, the result is the same as if we had conditioned on the smaller sigma-algebra $\mathcal{H}$ directly.

\section{MARTINGALES}
\noindent
We are given a probability space $(\Omega, \mathcal{F}, \mathbb{P})$ and a filtration as defined below. 
\begin{defn}
    A filtration $({\mathcal{F}_t})_{t\geq 0}$ is an increasing family of $\sigma$-algebras: $$\mathcal{F}_s \subseteq \mathcal{F}_t \subseteq \mathcal{F} \text{ for } s \leq t$$
\end{defn}
\noindent
From this we can define adaptation. 
\begin{defn}
    A sequence of random variables $\left\{X_n ; n \geq 0\right\}$ is adapted if:$$ X_n \in \mathcal{F}_n$$
\end{defn}
In Adaptation, the data $X_n$ only depends on information until instant $n$.
\begin{defn}
    We consider a sequence of random variables $X=\left\{X_n ; n \geq 0\right\}$ such that
    \begin{itemize}
        \item $\left\{X_n ; n \geq 0\right\}$ is adapted.
        \item $X_n \in L^1(\Omega)$ for all $n \geq 0$.
    \end{itemize} 
    Then
    \begin{itemize}
        \item $X$ is a martingale if $X_n=\mathrm{E}\left[X_{n+1} \mid \mathcal{F}_n\right]$.
        \item $X$ is a supermartingale if $X_n \geq \mathrm{E}\left[X_{n+1} \mid \mathcal{F}_n\right]$.
        \item  $X$ is a submartingale if $X_n \leq E\left[X_{n+1} \mid \mathcal{F}_n\right]$.
    \end{itemize}
\end{defn}
We see that the definition of a martingale was shared in definition $1.1$, this is the martingale property. This equation emphasizes the current value, $X_n$ is the expected value of the next step, given the information up to $n$. However, this definition is in discrete time, meaning that it follows the adaptation and integrability condition. To better understand this idea, we can use an example of  the fair game shared in the introduction. Using the same rules, we can now prove this definition. 

\begin{ex}
    Let $X_n$ represent the player's total winnings after $n$ games. We will show that $({X_n})_{n \geq 0}$ is a martingale. \\ Filtration:\\
    Define the filtration $\left\{\mathcal{F}_n\right\}_{n \geq 0}$ where $\mathcal{F}_n$ represents the information up to the $n$th game.
    Specifically, $\mathcal{F}_n$ contains the outcomes of the first $n$ coin tosses.\\
    Martingale Property:\\
    To prove that $\left\{X_n\right\}_{n \geq 0}$ is a martingale with respect to $\left\{\mathcal{F}_n\right\}_{n \geq 0}$, we need to verify the three conditions:
    \begin{itemize}
        \item Adaptation:
        $X_n$ is $\mathcal{F}_n$-measurable because the total winnings after $n$ games depend only on the outcomes of the first $n$ games.
        \item Integrability:
        The expectation $\mathbb{E}\left[\left|X_n\right|\right]$ is finite. Since $X_n$ is the sum of $n$ independent bets of $\pm 1$, its absolute value grows linearly with $n$, ensuring finite expectation.
        \item Martingale Property:
        $$
        \mathbb{E}\left[X_{n+1} \mid \mathcal{F}_n\right]=X_n
        $$
    \end{itemize} 
\begin{proof}
    Given that $X_n$ is the player's total winnings after $n$ games, the total winnings after $n+1$ games, $X_{n+1}$, can be written as:
    $$
    X_{n+1}=X_n+Y_{n+1}
    $$
    where $Y_{n+1}$ is the result of the $(n+1)$ th game
    \begin{itemize}
        \item $Y_{n+1}=1$ if the player wins the $(n+1)$ th game (heads).
        \item $Y_{n+1}=-1$ if the player loses the $(n+1)$ th game (tails).
    \end{itemize}
    The expected value of $Y_{n+1}$ given $\mathcal{F}_n$ is:
    $$
    \mathbb{E}\left[Y_{n+1} \mid \mathcal{F}_n\right]=0
    $$
    because the coin toss is fair, and the expected outcome of a fair coin toss is zero (since
    $$
    \left.P\left(Y_{n+1}=1\right)=P\left(Y_{n+1}=-1\right)=1 / 2\right)
    $$
    \noindent
    Therefore,
    $$
    \mathbb{E}\left[X_{n+1} \mid \mathcal{F}_n\right]=\mathbb{E}\left[X_n+Y_{n+1} \mid \mathcal{F}_n\right]=X_n+\mathbb{E}\left[Y_{n+1} \mid \mathcal{F}_n\right]=X_n+0=X_n
    $$
    \noindent
    This confirms that the process $\left\{X_n\right\}_{n \geq 0}$ is a martingale with respect to the filtration $\left\{\mathcal{F}_n\right\}_{n \geq 0}$
\end{proof}
\end{ex}
\noindent
From this example, we can move to another stochastic process like martingales, stopping time.
\begin{defn}
    Given a probability space $(\Omega, \mathcal{F}, \mathbb{P})$ and a filtration $\left\{\mathcal{F}_t\right\}_{t \geq 0}$, a random variable $\tau: \Omega \rightarrow[0, \infty]$ is a stopping time if for every $t \geq 0$,
    $$
    \{\tau \leq t\} \in \mathcal{F}_t .
    $$
\end{defn}
\begin{lemma}
    Let T be a stopping time If there exists a not random variable, $N$ and $ \varepsilon > 0$ such that $\mathbb{P}\left(T \leq n+N \mid \mathcal{F}_n\right)>\varepsilon \quad \forall n \geq 0$, then $\mathbb{E} [T]<\infty$.
\end{lemma}
\begin{proof}
    $\mathbb{P}(T > kN) = \mathbb{P}(T > k N \cup T > (k - 1) N)$ $$ = \mathbb{P}(T > kN | T > (k - 1) N) * \mathbb{P}(T > (k - 1)N) )$$ $$\leq (1-\varepsilon) * \mathbb{P}(T > (k - 1) N) \quad \leq \quad (1-\varepsilon) * \mathbb{P}(T > (k - 2) N)  \leq \cdots 
     \leq (1-\varepsilon)^{k}$$ $$\mathbb{E}[T] = \sum_{l = 0}^\infty \mathbb{P}(T > l) \leq N \sum_{k = 0} ^\infty \mathbb{P}(T > k N) \leq N \sum_{k = 0} ^\infty (1-\varepsilon)^k = \frac{N}{1-(1-\varepsilon)} = \frac{N}{\varepsilon}  < \infty $$
\end{proof}
\noindent
This shows that the expectation of a $T$, a stopping time is finite. Now, we can move into the martingale transform.

\begin{defn}
    We assume we have a probability space with a filtration. $C_n$ is predictable if $C_n$ is $\mathcal{F}_{n-1}$ measurable $\forall n$. Imagine that $C_n$ is the strategy that you are making at time $n$, but to make prediction at time $n$, we used the information at time $n-1$. Once we have this base information, we can move to the martingale transform of $X$. $$Y_n = \sum_{1 \leq k \leq n} C_k * (X_k - X_{k-1}) $$
\end{defn}
\begin{thm}
    If $C_n$ is predictable, and $0 \leq C_n \leq k$. $\forall (n,\omega)$ and X is a supermartingale, then so is the martingale transform, $C \cdot X$. If $C_n$ is predictable, and $|C_n| \leq K$ $\forall (n,\omega)$ and $X$ is a martingale, then so is $C \cdot X$
\end{thm}
\begin{rem}
    If $C_n$ has $\mathbb{E}[C_n^2] < \infty$, and $\mathbb{E}[X_n^2] < \infty$ $\forall n$, then Theorem $5.7$ still works without the bound, $|C_n| \leq K$.
\end{rem}
\begin{proof}
    To prove $Y$ is a martingale, then we must prove the following equation: $\mathbb{E}[(Y_n - Y_{n-1}) \mathcal{F}_{n-1}] = 0$. To prove this, we must find what the difference of $(Y_n - Y_{n-1})$ is and it should equal $C_n * (X_n - X_{n-1})$, where $C_n$ is $\mathcal{F}_{n-1}$ measurable. Let's begin by factoring out the original equation. $$ = C_n \mathbb{E}(X_n - X_{n-1}|\mathcal{F}_{n-1})$$ $$ = C_n(\mathbb{E}(X_n|\mathcal{F}_{n-1}) - \mathbb{E}(X_{n-1}|\mathcal{F}_{n-1}))$$ If $X$ is a martingale, $X_{n-1} = 0$, which is proved here. 
\end{proof}
\begin{thm}
    Let $T$ be a stopping time, if $X$ is a supermartingale, then so is $X^T$.
\end{thm}
\begin{proof}
    $C_n := \mathbbm{1}(n \leq T) = \mathbbm{1}(n-1 < T)$, which is $\mathcal{F}_n$ measurable. $$(C \cdot X)_n = \sum_{k=1}^n C_k (X_k - X_{k-1})$$ $$=\left\{\begin{array}{l}T \geq n: \sum_{k=1}^n\left(x_{k}-x_{k-1}\right)=x_n-x_0 \\ T<n: \sum_{k=1}^T\left(x_k-x_{k-1}\right)=x_T-x_0\end{array}\right.$$ \bigskip$$= X_{T \wedge n}-X_0 = {X^T}_n - X_0$$
\end{proof}
Now we can move to the Optional Stopping Theorem, the relation between stopping time and martingales. 
\begin{rem}
    If $M_n$ is a martingale, then $\mathbb{E}[M_n] = \mathbb{E}(\mathbb{E}[M_n|\mathcal{F}_{n-1}]) = \mathbb{E}[M_{n-1}] = \cdots = \mathbb{E}[M_0]$ (from the Tower Property).
\end{rem}
\begin{rem}
    If $T$ is a stopping time, $M_n$ is a martingale, then $\mathbb{E}[M_{T \wedge n}] = \mathbb{E}_{M_0}$.
\end{rem}
\begin{thm}
    Doob's Optional Stopping: Let $T$ be a stopping time, $X$ is a supermartingale. If it follows these conditions: 
    \begin{itemize}
        \item T is bounded OR X is bounded and T is infinite almost surely OR $\mathbb{E}[T] < \infty$ and $|X_n - X_{n-1}| \leq k$
    \end{itemize}
    Then $\mathbb{E}[X_T] \leq \mathbb{E}[X_0]$
\end{thm}
\begin{proof}
     If $T$ is not almost surely bounded, approximate it by a sequence of bounded stopping times $T_n=\min (T, n)$. Then we can use the martingale property to show that for each $T_n, \mathbb{E}\left[X_{T_n}\right]=\mathbb{E}\left[X_0\right]$ Finally, apply the Dominated Convergence Theorem  to take the limit as $n \rightarrow \infty$ and conclude that $\mathbb{E}\left[X_T\right]=$ $\mathbb{E}\left[X_0\right]$    
\end{proof}
After talking a bit about convergence in supermartingales, we can move to understanding Doob's Forward Convergence Theorem. 

\begin{defn}
    $X \in L_p$, if $\mathbb{E}[|X|^p] < \infty$; $|X|^p := (\mathbb{E}[|X|^p])^{1/p}$, where p is positive. 
\end{defn}
\begin{defn}
    $X_n$ is bounded in $L^p$, if we take $sup_n{||X_n||_p < \infty}$ This means if we take the p-norms, it will still be bounded regardless if we take the largest value of n. However, there is an uniform bound, $||X_n||_p \leq k < \infty$ for all $n$. 
\end{defn}
\begin{lemma}
    Using Doob's upcrossing Lemma, if $X_n$ is a supermartingale, then $(b-a) \mathbb{E}[U_n]$ on the interval $[a,b] \leq \mathbb{E}[X_n - a]^-$, we can do a collarary. 
\end{lemma}

\begin{cor}
    Let $X$ be a supermartingale, bounded in $L^1: sup_n \mathbb{E}[|X_n|] < \infty $. If we fix $a < b$, then $(b-a) \mathbb{E}[U_\infty]$ on the interval $[a,b] \leq |a| + sup_m \mathbb{E}[|X_m|] < \infty$. In particular, $\mathbb{P}(U_\infty [a,b] = \infty) = 0$.
\end{cor}
\begin{proof}
    $(b-a) \mathbb{E}[U_n][a,b] \leq \mathbb{E}[X_n - a]^- \leq \mathbb{E}|X_n - a| \leq \mathbb{E}|X_n| + |a|\leq sup_n \mathbb{E}|X_n| + |a|$. $U_n [a,b]$ is defined as number of upcrossings until time N, so if we know that it is non-decreasing in N, then it has a limit. $lim_{n \to \infty}: U_\infty [a,b]$. We need to make sure that this expression is finite, proving the corollary. To begin this, we take $(b-a) \mathbb{E}|U_\infty|[a,b]$ then use the monotone convergence theorem to find the limit. This means that it equals to $(b-a) lim_{n \to \infty } \mathbb{E} U_n [a,b]$, which is $\leq k + |a|$ due to the first statement of the proof. 
\end{proof}
\noindent
Now, we can move to what Doob's forward convergence theorem states. 

\begin{thm}
    Doob's forward convergence: Let $X$ be a supermartingales, bounded in $L^1$. Then $X_\infty := lim_n X_n$ exists almost surely and is almost surely finite. 
\end{thm}
\begin{proof}
    We want to prove that $X_n$ does not convergence. To prove that $X_\infty := lim_n X_n$ exists almost surely, we begin that $= \bigcup_{a < b} lim inf X_n < a < b < lim sup X_n \subseteq \bigcup_{a < b} U_\infty [a,b] = \infty $, so the probability of this latter part $= 0$. $$\mathbb{P}(X_n  \text { does not converge}) = 0$$ For the second part of the theorem, of "almost surely finite," we state that $\mathbb{E}|X_\infty| = \mathbb{E} [\text {lim inf}]|X_n| \leq$ (using Fatou's Lemma) $[\text {lim inf}] \mathbb{E}|X_n| \leq sup \mathbb{E}|X_n| < \infty$. Therefore, knowing that X is a supermartingale is bounded in $L^1$, we get $|X_\infty| < \infty $ almost surely.
\end{proof}
\begin{rem}
    If $X_n \geq 0$ is a supermartingale, then the $L^1$ bound is automatic: $\mathbb{E}|X_n| = \mathbb{E}[X_n] \leq \mathbb{E}[X_0] < \infty$
\end{rem}
\noindent
From here we are able to understand convergence in $L^p$, but take a case of a $L^2$ martingale. 
\noindent
$L^2$ martingales: $M_n \in L^2 \quad \forall n$. \\
\noindent
Then $M_v - M_u \perp L^2 (\mathcal{F}_u) $ In this situation, what does orthogonal exactly mean? Let's assume we have a square integrable random variable that is finite, we can define a similar scalar product. $\mathbb{E}(X \cdot Y)="\langle x, y\rangle"$ In the sense that these random variable are almost like vectors, we can fully understand what orthogonality means. Going back to the first first expression, $\forall v > u > t > s$, then following is true, $\mathbb{E}[(M_t - M_s) * (M_v - M_u)] = 0$ This is saying that the increment of my martingale from s to t and u to m are orthongal in the sense that my scalar product is 0. The proof of this is the following. 
\begin{proof}
    $\mathbb{E}[\mathbb{E} [(M_t - M_s) * (M_v - M_u) ] | \mathcal{F}_u]$ Because $v > u > t > s$, $M_t - M_s$ are $\mathcal{F}$-measurable. That means our new expression is, $\mathbb{E}[(M_t - M_s) * \mathbb{E}[(M_v - M_u) ] | \mathcal{F}_u]$ And by the martingale property, our latter expectation equals 0, this implies a theorem.
\end{proof}
\begin{thm}
    An $L^2$-martingale $M_n$ is bounded in $L^2$ if. $\sum_{k=1}^{\infty} \mathbb{E}\left(M_{k}-M_{k-1}\right)^2<\infty$. In this case, $M_n \rightarrow M_0$ almost surely in $L^2$. 
\end{thm}
\begin{proof}
    We can show proof of the Pythagorean Theorem. $\mathbb{E}[M^2_n] = ||M_n||^2_2$ $$= \mathbb{E} [M_0 + \sum_{k=1}^{\infty}(M_k - M_{k-1})]^2 $$ $$= \mathbb{E}[M^2_0] + \sum_{k=1}^n \mathbb{E}[M_k - M_{k-1}]^2 + 0$$
    Therefore, $M$ is bounded in $$L^2: \mathbb{E}[M_n^2] = \mathbb{E}[M_0^2] + \sum_{k=1}^n \mathbb{E}[M_k - M_{k-1}]^2 \leq \mathbb{E}[M_0^2] + \sum_{k=1}^\infty \mathbb{E}[M_k - M_{k-1}]^2 < \infty$$ \\
    Also, bounded in $L^1$ because of Doob's Forward Convergence, meaning that $$M_n \to^{\text {almost surely}} M_\infty$$ 
    The second norm states the following. $$\mathbb{E}[M_\infty - M_n]^2 \leq \text {lim inf} \mathbb{E}[M_{n+r} - M_n]^2 = \text {lim inf} \sum_{k =n+1}^{n+r} \mathbb{E}[M_k - M_{k-1}]^2 = \sum_{k =n+1}^{\infty} \mathbb{E}[M_k - M_{k-1}]^2 $$ and as $n$ approaches $\infty$, this expression equals 0.
\end{proof}

\section{APPLICATIONS IN FINANCE}
This martingale theory can be applied to stock market strategies. For instance, the Martingale Strategy suggests doubling your investment when you incur a loss. In the context of the stock market, if you purchase a stock at 10 dollars and its price drops to 9 dollars, you would buy another share at the new price. If the stock price falls further to 8 dollars, you would buy two additional shares. Eventually, when the stock price rebounds to 10 dollars, you not only recover your initial investment but also make a profit from the shares purchased at the lower prices. This approach reflects the martingale theory’s essence, where the future expected value of the process remains equal to the present value, despite short-term fluctuations.
\section{ACKNOWLEDGEMENTS}
I want to thank my mentor, Paulina Paiz for the valuable feedback and tips throughout this process. I would like to also thank Simon Rubenstein-Salzedo for organizing the Euler Circle.


\begin{thebibliography}{999}

\bibitem{goel22}
  Ishaan Goel,
  \emph{INTRODUCTION TO MARTINGALES WITH AN APPLICATION IN FINANCE}.
  UChicago,
  2022.
\bibitem{sheffield}
    Sheffield,
    \emph{Martingales, risk neutral probability, and Black-Scholes option pricing}.
    MIT.
\bibitem{samy}
    Dr. Samy Tindel,
    \emph{Martingales}.
    Purdue University.
\bibitem{}
    David Williams, 
    \emph{Probability with Martingales}.
    Cambridge University Press,
    1991.
\end{thebibliography}
\end{document}